\documentclass[11pt]{amsart}

\usepackage{enumerate,url,amssymb,mathrsfs,upref,nicefrac}

\newtheorem{theorem}{Theorem}[section]
\newtheorem{lemma}[theorem]{Lemma}
\newtheorem{proposition}[theorem]{Proposition}
\newtheorem{corollary}[theorem]{Corollary}

\theoremstyle{definition}

\newtheorem{example}[theorem]{Example}

\theoremstyle{remark}
\newtheorem{remark}[theorem]{Remark}

\numberwithin{equation}{section}

\newcommand{\be}[1]{\begin{equation}\label{#1}}
\newcommand{\ee}{\end{equation}}
\newcommand{\abs}[1]{\lvert#1\rvert}

\newcommand{\norm}[1]{\lVert#1\rVert}

\newcommand{\C}{\mathbb{C}}

\newcommand{\R}{\mathbb{R}}
\newcommand{\T}{\mathbb{T}}

\newcommand{\Hp}{\mathbb{H}}

\DeclareMathOperator{\dist}{dist}
\DeclareMathOperator{\re}{Re}
\DeclareMathOperator{\im}{Im}
\DeclareMathOperator{\diam}{diam}
\DeclareMathOperator{\sgn}{sign}

\begin{document}

\title{Sharp distortion growth for bilipschitz extension of~planar~maps}

\author[L. V. Kovalev]{Leonid V. Kovalev}
\address{Department of Mathematics, Syracuse University, Syracuse,
NY 13244, USA}
\email{lvkovale@syr.edu}
\thanks{Supported by the NSF grant DMS-0968756.}



\begin{abstract}
This note addresses the quantitative aspect of the bilipschitz extension problem.   
The main result states that any bilipschitz embedding of $\mathbb R$ into $\mathbb R^2$ can be extended 
to a bilipschitz self-map of $\mathbb R^2$ with a linear bound on the distortion.
\end{abstract}

\maketitle

\section{Introduction}

Let $X$ and $Y$ be metric spaces, with distance denoted by $\abs{\cdot}$. 
A map $f\colon X\to Y$ is an $(L,\ell)$-bilipschitz embedding if 
\be{defbl}
\ell \abs{a-b} \le \abs{f(a)-f(b)} \le L \abs{a-b}\qquad \text{for all} \qquad a,b\in X.
\ee
If $f$ is in addition surjective, it is a bilipschitz isomorphism between $X$ and $Y$ (an automorphism if 
$X=Y$). In this paper $X$ and $Y$ will be subsets of the plane $\R^2\simeq \C$.
The following theorem was proved by Tukia in~\cite{Tu,Tu2}, see also~\cite{JK, Lat}. 

\begin{theorem}\label{halfplane1}
Any $(L,\ell)$-bilipschitz embedding $f\colon \R\to\C$ can be extended to
an $(L',\ell')$-bilipschitz automorphism $F\colon \C\to\C$, where $L',\ell'$ depend only on $L$ and $\ell$.
\end{theorem}

Conjugating $f$ by a M\"obuis map~\cite[p.~93]{Tu2}, one obtains a version of  Theorem~\ref{halfplane1} for bilipschitz embeddings $f\colon \T\to\C$, where $\T$ is the unit circle $\{\abs{z}=1\}$.  The aforementioned papers do not provide explicit estimates for $L'$ and $\ell'$, although such  estimates could in principle be obtained by following the approach of~\cite{JK,Tu2}.

For embeddings of $\T$, Daneri and Pratelli~\cite{DP} recently achieved $L'=CL^4$ and $\ell'=\ell^4/C$
with a universal constant $C$. They asked whether $L'$ and $\ell'$  can have linear growth, that is
$L'=CL$ and $\ell'=\ell/C$ with universal $C$. We prove this for embeddings of $\R$.

\begin{theorem}\label{halfplane}
Any $(L,\ell)$-bilipschitz embedding $f\colon \R\to\C$ can be extended to
an $(L',\ell')$-bilipschitz automorphism $F\colon \C\to\C$, where $L'=2000L$ and $\ell'=\ell/120$.
\end{theorem}

 Theorem~\ref{halfplane} appears to be the first bilipschitz extension result with linear growth of distortion.  This
contrasts with  the state of our knowledge of Lipschitz extensions, for which the linear bound $L'\le CL$ is usually available~\cite{BB}. 
In the Euclidean setting one  even has a Lipschitz extension with $L'=L$, by Kirszbraun's theorem~\cite{Ki}.

In the case of small distortion, $L/\ell \approx 1$, Tukia and V\"ais\"al\"a~\cite[Theorem 5.3]{TV2} used a triangulation method to construct  an extension such that $L'/\ell'\to 1$ as $L/\ell\to 1$; moreover, their construction works in all dimensions. In contrast to linear
\emph{growth} in Theorem~\ref{halfplane}, the \emph{decay} of $L'/\ell'-1$ cannot be linear. The best possible estimate is of order 
$(L/\ell-1)^{1/2}$~\cite[Example  5.12]{TV2}.  The sets that allow for extensions with linear decay of the distortion were recently 
studied in~\cite{ATV,ATV2,T}. The extension used in the proof of Theorem~\ref{halfplane} does not have the 
property $L'/\ell'\to 1$ as $L/\ell\to 1$. For example, the identity map is extended to $F(x+iy)=x+iy/2$. It is conceivable that one could achieve $L'/\ell'\to 1$ with a modified construction and more elaborate estimates, but this would come at the cost of a substantially longer proof (cf.~\cite{BA}). 

In order to achieve the linear growth bound, the proof of Theorem~\ref{halfplane} must be structured so that 
the bilipschitz property of $f$ is used only once. Thus the proof is necessarily different from the earlier ones.
For instance, it does not rely on the fact that conformal maps between quasidisks are quasisymmetric. Instead,
it is based on the Beurling-Nevanlinna theorem for harmonic measure~\cite{GMb,Ranb}.
The other key tools are the Riemann mapping theorem and the Beurling-Ahlfors extension, as in~\cite{Tu2}.

Conjugating $f$ and $F$ by a global bilipschitz map yields a corollary for unbounded chordarc curves, i.e.,  
images of $\R$ under bilipschitz embeddings into $\mathbb C$. 

\begin{corollary}
For any unbounded chordarc curve $\Gamma\subset \C$ there exists a constant $C$ such that
any $(L,\ell)$-bilipschitz embedding $f\colon \Gamma\to\C$ can be extended to
a $(CL,\ell/C)$-bilipschitz automorphism $F\colon \C\to\C$.
\end{corollary}

\begin{remark}
In the literature it is common to use $1/L$ in place of $\ell$ in the formula~\eqref{defbl} and call $f$ an $L$-bilipschitz map.  
This introduces a symmetry between expansion and contraction, between $f$ and $f^{-1}$. However, such symmetry is 
lacking in the setting of Theorem~\ref{halfplane}, where the domain of $f$ is a line while the range can be any chordarc curve. 
Thus, one cannot expect the expansion and contraction constants to deteriorate by the same factor in the process of extension. 
Indeed, in Theorem~\ref{halfplane} we have two different factors, $2000$ and $120$, attached to the expansion and compression constants. On the other hand, these factors are probably far from optimal (cf. Example~\ref{ex2}), and so their lack of equality
does not mean much. 
\end{remark} 

Due to topological obstructions, Theorem~\ref{halfplane1} does not hold in higher dimensions~\cite{Tu}. 
However, there are several bilipschitz extension results that avoid or  overcome such obstructions~\cite{AS,DS,Ka,Lat2,Lu,Mac,Sem,TV1,TV3,Va}  
and it would be interesting to know which of them admit linear control of the distortion. The original question
of Daneri and Pratelli, for embeddings of $\T$ into $\C$, also remains open.  

The paper is structured as follows. 
Section~\ref{prelim} collects prerequisite results. In~\S\ref{conf} we estimate the
derivative of a conformal map of halfplane. Bilipschitz condition makes its first appearance in \S\ref{mainproof}
where Theorem~\ref{halfplane} is proved. In \S\ref{examples} we show that some increase of the distortion 
is inevitable, i.e., one cannot have $(L',\ell')=(L,\ell)$.

\subsection*{Acknowledgements}
This paper benefited from discussions at AIM workshop ``Mapping theory in metric spaces'' organized by
Luca Capogna, Jeremy Tyson and Stefan Wenger. In particular,  thanks are due to Mario Bonk,
Melkana Brakalova, Colin Carroll, Hrant Hakobyan, Kabe Moen, and Charles Smart. 

\section{Preliminaries}\label{prelim}

Throughout the paper $\Hp$ denotes the open upper halfplane. The open disk of radius $r$ centered at $a$
is denoted by $B(a,r)$ and its closure by $\overline{B}(a,r)$. We write $\omega(\zeta,E,\Omega)$
for the harmonic measure of a set $E\subset\partial \Omega$ with respect to a point $\zeta$ in a domain $\Omega$.

The following proposition is~\cite[Corollary 4.5.9]{Ranb}; it can be seen as a corollary of
the Beurling-Nevanlinna projection theorem.

\begin{proposition} Let $\Omega\subset \C\setminus \{0\}$ be a simply connected domain. Pick a point
$\zeta\in\Omega$ and let $\rho>0$.

(a) If $\abs{\zeta}<\rho$, then
\be{BN1}
\omega(\zeta,\partial \Omega\cap B(0,\rho), \Omega)\ge \frac{2}{\pi}\sin^{-1}\left(\frac{\rho-\abs{\zeta}}{\rho+\abs{\zeta}}\right).
\ee

(b) If $\abs{\zeta}>\rho$, then
\be{BN}
\omega(\zeta,\partial \Omega\cap \overline{B}(0,\rho), \Omega)\le \frac{2}{\pi}\cos^{-1}\left(\frac{\abs{\zeta}-\rho}{\abs{\zeta}+\rho}\right).
\ee
\end{proposition}

In~\cite{Ranb} the inequality~\eqref{BN} is stated for the open disk $B(0,\rho)$. The version
given here follows by replacing $\rho$ with $\rho+\epsilon$ and letting $\epsilon\to 0$.

The following two-sided derivative bound is a straightforward consequence of the Schwarz lemma and
the Koebe $\nicefrac{1}{4}$-theorem.

\begin{proposition} Let $\Phi\colon \Hp\to\Omega\subset \C$ be a conformal map. For any $z\in \Hp$
\be{rad}
\frac{1}{2}\im z \,\abs{\Phi'(z)} \le \dist( \Phi(z), \partial\Omega) \le 2 \im z \,\abs{\Phi'(z)}.
\ee
\end{proposition}

\section{Distortion of a conformal map}\label{conf}

\begin{lemma}
Suppose that $\Phi\colon \Hp\to\Omega$ is a
conformal map that extends continuously to the boundary $\partial \Hp=\R$.
Let $\phi\colon \R\to\partial\Omega$ denote the induced boundary map.
Fix $x+iy\in \Hp$ and consider four subsets of $\partial \Omega$ defined as
\be{arcs}
\begin{split}
\Gamma_1=\phi((-\infty,x-y]),\qquad &\Gamma_2=\phi([x-y,x-y/2]), \\
\Gamma_3=\phi([x+y/2,x+y]), \qquad &\Gamma_4=\phi([x+y,+\infty)).
\end{split}
\ee
Then
\be{harm1}
\frac{1}{120} \dist (\Gamma_1,\Gamma_4) \le y \abs{\Phi'(x+iy)}\le
500 \min \left( \diam \Gamma_2,\diam \Gamma_3 \right).
\ee
\end{lemma}

\begin{proof} Denote $\zeta=\Phi(x+iy)$.
Ii order to prove the first inequality in~\eqref{harm1} we first show that
\be{h1}
\dist( \zeta,\Gamma_j) \le 30 \dist( \zeta,\partial\Omega),\qquad j=1,4.
\ee
By symmetry it suffices to consider $j=1$.
Recall that the harmonic measure $\omega$ in the halfplane coincides with the normalized angular measure~\cite[I.1]{GMb}. This and the conformal invariance of the harmonic measure imply
\be{angle1}
\omega(\zeta,\Gamma_1,\Omega)
= \omega(x+iy, (-\infty,x-y],\Hp) = \frac{1}{\pi}\frac{\pi}{4} = \frac{1}{4}.
\ee
By translation we can achieve $0\in \partial \Omega$ and $\abs{\zeta}=\dist(\zeta,\partial \Omega)$.
Let $\rho=\dist(\zeta ,\Gamma_j)$.  If $\rho=\abs{\zeta}$ then~\eqref{h1} holds. Otherwise $\rho>\abs{\zeta}$
and we can apply~\eqref{BN1} to obtain
\be{h2}
\omega(\zeta,\partial \Omega\cap B(0,\rho), \Omega)\ge \frac{2}{\pi}\sin^{-1}\left(\frac{\rho-\abs{\zeta}}{\rho+\abs{\zeta}}\right).
\ee
On the other hand, $\partial \Omega\cap B(0,\rho)$ is disjoint from $\Gamma_1$, which yields
\be{h3}
\omega(\zeta,\partial \Omega\cap B(0,\rho), \Omega) \le 1-\omega(\zeta,\Gamma_1,\Omega) =\frac{3}{4}.
\ee
Combining~\eqref{h2} and~\eqref{h3} we find that
\[
\frac{\rho-\abs{\zeta}}{\rho+\abs{\zeta}} \le \sin \frac{3\pi }{8}, \qquad \text{hence}\qquad \rho\le 30\abs{\zeta}
\]
which is~\eqref{h1}.

From~\eqref{h1} and~\eqref{rad} we obtain
\[
\dist(\Gamma_1,\Gamma_2) \le 60 \dist( \zeta,\partial\Omega) \le 120\, y\, \abs{\Phi'(x+iy)}
\]
completing the proof of the first inequality in~\eqref{harm1}.

To establish the second part of~\eqref{harm1} it suffices to consider $\Gamma_2$, as $\Gamma_3$ can be
treated similarly.  We again use the invariance of harmonic measure,
\be{angle2}
\omega(\zeta,\Gamma_2,\Omega) = \omega(x+iy, [x-y,x-y/2],\Hp) = \frac{1}{\pi}\left(\frac{\pi}{4}-\frac{\pi}{6}\right) = \frac{1}{12}.
\ee
By translation we can assume $0\in \Gamma_2$. Let $\rho=\diam \Gamma_2$.  By~\eqref{angle2}
\[
\omega(\zeta,\partial \Omega\cap \overline{B}(0,\rho), \Omega) \ge \frac{1}{12}.
\]
Comparing the latter bound with~\eqref{BN} we conclude that
\[
\abs{\zeta} \le \rho \, \frac{1+\cos \frac{\pi}{24} }{1-\cos \frac{\pi}{24}}<250\, \rho
\]
By virtue of~\eqref{rad},
\[
y \abs{\Phi'(x+iy)} \le 2\dist (\zeta, \partial \Omega) \le 2\abs{\zeta}<500 \diam \Gamma_2
\]
as required.
\end{proof}

\section{Proof of Theorem~\ref{halfplane}}\label{mainproof}

Let $DF$ denote the Jacobian matrix of $F$.
Our goal is to construct a homeomorphism $F\colon \C\to\C$ that agrees with $f$ on the real axis $\R$,
is $C^1$-smooth in $\C\setminus \R$, and satisfies the inequalities
\be{lip1}
\norm{DF(z)}\le 2000L \qquad \text{and} \qquad \norm{(DF(z))^{-1}}\le 120/\ell
\ee
for all $z\in \C\setminus \R$. Indeed, the desired Lipschitz properties of $F$ and $F^{-1}$ follow
from~\eqref{lip1} by integration along line segments.

We focus on extending $f$ to the upper halfplane $\Hp$, the extension to the lower halfplane being similar.
The unbounded simple curve $\Gamma=f(\R)$ divides the plane into two unbounded domains.
One of them, denoted $\Omega$, corresponds to $\Hp$ via the boundary orientation induced by $f$.
Let $\Phi$ be a conformal map of $\Hp$ onto $\Omega$ such that $\abs{\Phi(z)}\to\infty$ as $\abs{z}\to\infty$.
By Carath\'eodory's theorem $\Phi$ extends to a homeomorphism between the closures of $\Hp$ and $\Omega$.
We use the lowercase letter $\phi$ for the boundary correspondence $\phi\colon \R\to \Gamma$.

Define $\psi \colon \R\to\R$ by $\psi=f^{-1}\circ \phi$; this is an increasing homeomorphism of $\R$ onto $\R$.
The bilipschitz property of $f$ can now be stated as
\begin{equation}\label{bla1}
\ell \abs{\psi(a)-\psi(b)} \le  \abs{\phi(a)-\phi(b)} \le  L\abs{\psi(a)-\psi(b)}, \qquad a,b\in \R.
\end{equation}
For future reference we note two  consequences of~\eqref{bla1}:
\be{diam1}
\diam \phi([a,b]) \le L (\psi(b)-\psi(a)) \qquad \text{whenever }\ a<b,
\ee
and
\be{dist1}
\dist (\phi((-\infty,a]), \phi([b,+\infty))  \ge \ell (\psi(b)-\psi(a)) \qquad \text{whenever }\ a<b.
\ee

Let $\Psi\colon \Hp\to\Hp$ denote the Beurling-Ahlfors extension of $\psi$, namely
\begin{equation}\label{ext1}
\Psi(x+iy)=\frac{1}{2}\int_{-1}^{1} \psi(x+ty)(1+i\sgn t)\,dt.
\end{equation}
This is a quasiconformal map of $\Hp$ onto itself~\cite{Ahb,BA}.  The desired extension $F\colon \Hp\to\Omega$
will be the composition $F=\Phi\circ \Psi^{-1}$. Thus, our goal~\eqref{lip1} can be restated in terms of $\Psi$ as follows.
\be{lip2}
\frac{\ell}{120}\norm{D\Psi (z)}\le \abs{\Phi'(z)} \le \frac{2000L}{\norm{(D\Psi(z))^{-1}}}.
\ee

We fix a point $z=x+iy\in \Hp$ for the rest of the proof.
The partial derivatives of $\Psi$ were computed in~\cite[p.~43]{Ahb}.
\be{dm1}
D\Psi = \frac{1}{2y}\begin{pmatrix}  \alpha+\beta & \gamma-\delta \\ \alpha-\beta & \gamma+\delta \end{pmatrix}
\ee
where
\be{pd2}
\begin{split}
\alpha=\psi(x+y)-\psi(x), \qquad
&\gamma = \int_0^1 (\psi(x+y)-\psi(x+ty))\,dt,
 \\
\beta = \psi(x)-\psi(x-y), \qquad
&\delta = \int_{-1}^0 (\psi(x+ty)-\psi(x-y))\,dt
\end{split}
\ee
Here $\alpha,\beta,\gamma,\delta>0$ because $\psi$ is increasing.
For the same reason, $\alpha>\gamma$ and $\beta>\delta$. It follows that the greatest entry of $D\Psi$ is $\frac{\alpha+\beta}{2y}$.
Hence
\be{ub0}
\norm{D\Psi} \le \frac{\alpha+\beta}{y} = \frac{\psi(x+y)-\psi(x-y)}{y}.
\ee
We now invoke~\eqref{dist1} and~\eqref{harm1} to obtain
\[
\norm{D\Psi} \le \frac{1}{y \ell} \dist(\Gamma_1,\Gamma_4) \le \frac{120}{\ell}\,\abs{\Phi'(x+iy)}
\]
which is the first inequality in~\eqref{lip2}.

To prove the second inequality in~\eqref{lip2}, we calculate the Jacobian determinant
\be{jd}
\det D\Psi=\frac{\alpha\delta +\beta\gamma}{2y^2}
\ee
and use~\eqref{jd} and \eqref{ub0} to estimate
\be{smallsigma}
\norm{(D\Psi)^{-1}}= \frac{\norm{D\Psi}}{\det D\Psi} \le
\frac{\alpha+\beta}{y}  \frac{2y^2}{\alpha\delta +\beta\gamma}
\le \frac{2y}{\min(\gamma,\delta)}.
\ee
A lower estimate on $\gamma$ and $\delta$ is now required. Recall that
\[
\gamma =\int_0^1 (\psi(x+y)-\psi(x+ty))\,dt \ge \frac{1}{2}(\psi(x+y)-\psi(x+y/2))\ge \frac{1}{2L}\diam \Gamma_3
\]
where the last step is based on~\eqref{diam1}. Similarly,
$\delta \ge \frac{1}{2L}\diam \Gamma_2$. The second inequality in~\eqref{harm1} yields
\[
\min(\gamma, \delta) \ge \frac{1}{1000L}\, y\, \abs{\Phi'(x+iy)}
\]
which we plug into~\eqref{smallsigma} to obtain
\[
\norm{(D\Psi)^{-1}} \le \frac{2000L}{\abs{\Phi'}}
\]
completing the proof of~\eqref{lip2} and of the theorem.
\qed

\begin{remark}\label{independent} The construction of $F$ involved a conformal map $\Phi$ which is not determined uniquely. Indeed, the 
requirement that $F$ maps $\Hp$ onto $\Omega$ while fixing the boundary point $\infty$ determines $F$ only up to 
pre-composition with a linear map $\eta=r z+ s$, with $r>0$ and $s\in\R$. However, replacing $\Phi$ with $\Phi\circ \eta$ 
results in replacing $\Psi$ with $\Psi\circ \eta$, and consequently, the composition $F=\Phi\circ \Psi^{-1}$ remains unchanged. 
Thus, the extended map $F$ depends only on $f$. 

In a similar way one observes that the extension process commutes with linear transformations, i.e., the extension of $f\circ \eta$ is $F\circ \eta$, and the extension of $\eta\circ f$ is $\eta\circ F$. However, this commutativity does not hold for general M\"obius maps due 
to the distinguished role of the point $\infty$. It is unclear whether the conformally natural extension constructed by Douady and Earle~\cite{DE} allows for a linear bound on bilipschitz distortion. 
\end{remark}

\section{Examples}\label{examples}

By Kirszbraun's theorem any $L$-Lipschitz map from a subset of $\R^n$ into $\R^n$ can be
extended to an $L$-Lipschitz map of $\R^n$. That is, no loss of the Lipschitz constant is incurred
in the extension. In this section we show that bilipschitz maps do not always
admit such a lossless extension. First consider the embeddings of $\T$.

\begin{example}\label{ex1}
Define $f\colon \T\to\C$ by
\begin{equation}\label{fold}
f(x+iy)=\abs{x+\nicefrac{1}{2}} +iy
\end{equation}
This map is $(1,\ell)$-bilipschitz for some $\ell\in (0,1)$. However, any bilipschitz extension
$F\colon\C\to\C$ must have $L'\ge \frac{2\pi}{3\sqrt{3}}>1$.
\end{example}

\begin{proof}
The line segment connecting the points $-\nicefrac{1}{2}\pm i\frac{\sqrt{3}}{2}$ is mapped by $F$ to a curve that connects the points $\pm i\frac{\sqrt{3}}{2}$ and stays within the domain enclosed by $f(\T)$. Such a curve
must have length at least $\nicefrac{2\pi}{3}$.
\end{proof}

The idea of Example~\ref{ex1} is that folding a curve $\Gamma$ across a line may increase the intrinsic metric in
one of the components of $\C\setminus \Gamma$. This trick does not work for $\Gamma=\R$,  but we have a different
example for this case, cf.~\cite[Example 5.12]{TV2}. 

\begin{example}\label{ex2}
Define $f\colon \R\to\C$ by
\begin{equation}\label{bend}
f(x)=\begin{cases} x\qquad & x\ge 0, \\ ix\qquad & x\le 0. \end{cases}
\end{equation}
This map is $(1,\frac{1}{\sqrt{2}})$-bilipschitz but it does not admit any $(1.1,0.7)$-bilipschitz extension
$F\colon\C\to\C$.
\end{example}

\begin{proof} Suppose $F$ exists and let $w=F(i)$. From the inequalities
\[
\abs{w-2}\le 1.1\abs{i-2}=1.1\sqrt{5}  \qquad \text{and}\qquad \abs{w+2i}\le 1.1\abs{i+2}=1.1\sqrt{5}
\]
we obtain $\re w\ge 2-1.1\sqrt{5}$ and $\im w\le 1.1\sqrt{5}-2$. Therefore, the distance from $w$
to $f(\R)$ does not exceed
$\sqrt{2}\,(1.1\sqrt{5}-2)<0.7$,
which is a contradiction because $\dist(i,\R)=1$.
\end{proof}

The  factor of $1.1$ in Example~\ref{ex2} is rather modest compared to $2000$ in Theorem~\ref{halfplane}.
This leaves the possibility that Theorem~\ref{halfplane} could hold with much better bounds, for example $L'=2L$ and $\ell'=\ell/2$.

\bibliographystyle{amsplain}

\end{document}